\documentclass{article}   	% use "amsart" instead of "article" for AMSLaTeX format

\usepackage{geometry, amsmath,mathrsfs}                		% See geometry.pdf to learn the layout options. There are lots.
\usepackage{graphicx}				% Use pdf, png, jpg, or eps§ with pdflatex; use eps in DVI mode
\usepackage{indentfirst}								% TeX will automatically convert eps --> pdf 
\usepackage{amssymb}
\usepackage{cite}
\usepackage[utf8]{inputenc}
\usepackage{amssymb}
\usepackage{tikz}
\usepackage{pgfplots}
\usepackage{cite}
\usepackage{bm}
\usepackage{bbold}
\usepackage[affil-it]{authblk}
\usepackage[colorlinks,
            linkcolor=blue,
            anchorcolor=green,
            citecolor=red
            ]{hyperref}

\allowdisplaybreaks

\numberwithin{equation}{section}
\newtheorem{theorem}{Theorem}[section]
\newtheorem{lemma}[theorem]{Lemma}
\newtheorem{proposition}[theorem]{Proposition}
\newtheorem{corollary}[theorem]{Corollary}
\newtheorem{definition}[theorem]{Definition}
\newtheorem{example}[theorem]{Example}

\newcommand{\Var}{\mathrm{Var}}

\newenvironment{proof}[1][Proof]{\begin{trivlist}
\item[\hskip \labelsep {\bfseries #1}]}{\end{trivlist}}
\newenvironment{remark}[1][Remarks]{\begin{trivlist}
\item[\hskip \labelsep {\bfseries #1}]}{\end{trivlist}}

\newcommand*{\QEDA}{\hfill\ensuremath{\blacksquare}}%

\begin{document}

\title{On the local times of stationary processes with  conditional local limit theorems}

\author{Manfred Denker \footnote{mhd13@psu.edu}, Xiaofei Zheng \footnote{xxz145@psu.edu}}
\affil{Department of Mathematics, The Pennsylvania State University, \\State College, PA, 16802, USA}
\date{}
\maketitle

\begin{abstract}
We investigate the connection between conditional local limit theorems and the local time of integer-valued stationary  processes. We show that a conditional local limit theorem (at $0$) implies the convergence of local times to Mittag-Leffler distributions, both in the weak topology of distributions and a.s. in the space of distributions.
\end{abstract}

\section{Introduction and Main Results}\label{sec:1}
Local time characterizes the amount of time a process spends at a given level.  Let $X_1, X_2, ...$ be  integer-valued random variables, $S_{n}=\sum_{i=1}^{n}X_i$.   The local time of $\{S_n\}$ on level $x$ at time $n$ is defined to be $\ell(n,x):=\#\{i=1, 2,\cdots, n: S_{i}=x\}$. Denote by $\ell_n$ the local time at $0$ at time $n$ for short.  For simple random walks, the exact and the limit distributions of ${\ell_n}$ are well  known. Chung and Hunt  (1949) \cite{HuntChung} studied the limit behavior of the sequence of $\{\ell_n\}$. R{\'e}v{\'e}sz (1981) \cite{Revesz} proved an almost sure invariance principle  by Skorokhod embedding.  For more general random walks, Borodin (1984) \cite{Borodin} established the weak convergence of $\ell(x\sqrt{n},[nt])/\sqrt{n}$ of a recurrent random walk to the Brownian local time. Ale{\v{s}}kevi{\v{c}}ien{\.e} (1986) \cite{AA} gave the asymptotic distribution and moments of local times of an aperiodic recurrent random walk. Bromberg and Kosloff (2012) \cite{BroKos} proved weak invariance principle for the local times of partial sums of Markov Chains. Bromberg (2014) \cite{MB} extended it to Gibbs-Markov processes.

In this article we study the connection of local times and local limit theorems as stated below.
\begin{definition}\label{clltdef}
A centered  integer-valued stationary  process $\{X_n\}$ is said to have  a conditional local limit theorem at $0$, if there exists a constant $g(0)>0$ and  a sequence $\{B_n\}$ of positive real numbers, such that  for all $x\in \mathbb Z$
\begin{equation}\label{cllt}
\lim_{n\to \infty}B_n P(S_n=x| (X_{n+1},X_{n+2},...)=\cdot)= g(0)\  \quad  \mbox{\rm }
\end{equation}
almost surely.
\end{definition}

The full formulation of the corresponding form of a local limit theorem goes back to Stone and reads in the conditional form (see \cite{Jon3}) that
\begin{equation}\label{full}
\lim_{n\to \infty}B_n P(S_n=k_n| (X_{n+1},X_{n+2},...)=\cdot)= g(\kappa)\  \quad  \mbox{\rm as $ \frac{k_n-A_n}{B_n}\to \kappa$ \ $P$-a.s.}
\end{equation}
for all $\kappa\in \mathbb R$, where $A_n$ is some centering constant. 
Condition (\ref{clltdef}) can be reformulated using the dual operator $P_T$ of the isometry $Uf=f\circ T$ on $L^\infty(P)$ operating on $L^1(P)$, where we take the $L^p$-spaces of $P$ restricted to the $\sigma$-field generated by all $X_n$ and $T$ the shift operation $T(X_1,X_2,...)=(X_2,X_3,...)$. This operator is called the transfer operator. The local limit theorem at $0$ then reads
\begin{equation}
\lim_{n\to \infty}B_n P_{T}^n(\mathbb 1_{\{S_n=x\}})=g(0)\;\qquad \text{\rm for all $x\in \mathbb Z$,\  $P$-a.s.}.
\end{equation}

In this paper, we assume that $\{X_n\}$ has the conditional local limit theorem at $0$ as formulated in (\ref{clltdef}).

\begin{remark}
For the full formulation (\ref{full}), if the convergence is uniformly for almost all $\omega$,
it  would imply that $\{X_n\}$ has a local limit theorem, then by \cite{IbragimovLinnik}, 
$\{X_n\}$ are in the domain of attraction of a stable law with index $d$:
$$\dfrac{S_n-A_n}{B_n}\stackrel{\mathcal{W}}{\longrightarrow}Z_d.$$
The  probability density function  of $Z_d$  is $g$ as above and the cumulative distribution function  is denoted by $G(x)$.
Since $\int_{\Omega}\phi ~dP=0$, we can (and will) assume that  $A_n=0$. It is necessary \cite{IbragimovLinnik} that $\{B_n\}$ is regularly varying of order $\beta=1/d$. 
\end{remark}

Next we state the main results of this paper.
\begin{theorem}[Convergence of local times]\label{main}
Suppose that the integer-valued stationary  process $\{X_n:=\phi \circ  T^{n-1}: n\geq 1\}$ has a conditional local limit theorem at $0$ (\ref{cllt}) with regularly varying scaling constants $B_n=n^{\beta}L(n)$, where $\beta\in [\frac{1}{2},1)$ and $L$ is a  slowly varying function. Put $a_n:=g(0)\sum_{k=1}^n\dfrac{1}{B_k}\to \infty$. Then $\dfrac{\ell_n}{a_n}$ converges to  a random variable $ Y_{\alpha}$ strongly in distribution, i.e.
\begin{eqnarray}
\int_{\Omega}g\bigg(\dfrac{\ell_n(\omega)}{a_n}\bigg)H(\omega)dP(\omega)\to E[g(Y_{\alpha})],
\end{eqnarray}
for any bounded and continuous function $g$ and any probability density function $H$ on $(\Omega, \mathcal{F}, P),$ and $Y_{\alpha}$ has the normalized  Mittag-Leffler distribution of order $\alpha=1-\beta$. 
\end{theorem}
\begin{remark}
In Theorem \ref{main}, $\beta\in [\frac{1}{2},1)$. It is because $\beta=\frac{1}{d}$, where $d\in (0,2]$ is the stability parameter of the stable law $Z_d$. Also, to ensure $a_n$ is divergent, $\beta$ has to be less than $1$.

Strong convergence in distribution is stronger than weak convergence. For the definitions of strong convergence in distribution and Mittag-Leffler distribution, we refer to \cite{Jon1} Sections 3.6 and 3.7. 
\end{remark}

The representation of the local time $\ell_n$ as ergodic sums in the proof  of Theorem \ref{main} enables us to apply the results in Aaronson and Denker (1990) \cite{Jon2} directly.  Estimates of the deviation and the upper bounds of local times of $\{S_n\}$ are given in the following theorem.

\begin{theorem}[Deviation and Upper bound]\label{deviation}
Suppose  $\{X_n\}$ satisfies all conditions in Theorem \ref{main}.    Then for every $\gamma>1$, there exists a constant $n_{\gamma}$ such that for all $n_{\gamma}\leq t\leq L_2(n)^2$, where $L_2(n)=\log \log n$, one has
\begin{equation}\label{dev}
e^{-\gamma(1-\alpha) t}\leq P(\ell_n\geq \dfrac{\Gamma(1+\alpha)}{\alpha^{\alpha}}ta_{n/t})\leq e^{-\frac{1}{\gamma}(1-\alpha)t}.
\end{equation}
In addition, if the return time process $R_n$ of $\ell_n$ is uniformly or strongly mixing from below \footnote{Refer to \cite{Jon2} for  uniformly or strongly mixing from below.}, where $R_n$ is the waiting time for $\ell_n$ to arrive at $0$ the $n$th time, then
\begin{equation}\label{sup}
\limsup_{n\to \infty} \dfrac{\ell_n}{a_{\frac{n}{L_2(n)}}L_2(n)}=K_{\alpha}, a.s.
\end{equation}
where $K_{\alpha}=\dfrac{\Gamma(1+\alpha)}{\alpha^{\alpha}(1-\alpha)^{1-\alpha}}$. 
\end{theorem}

The second part of the theorem was proved by Chung and Hunt in 1949 \cite{HuntChung} for simple random walks, by Jain and Pruitt \cite{Jain1} and Marcus and Rosen \cite{Marcus1} for more general random walks. 

\begin{corollary}[Gibbs-Markov transformation \cite{Jon3}] \label{Gibbs-Markov}
Let $(\Omega,\mathcal{B},P,T,\alpha)$ be a mixing, probability preserving Gibbs-Markov map (see \cite{Jon3} for definition), and let 
$\phi: \Omega \to \mathbb{Z}$
be  Lipschitz continuous on each $a\in \alpha$, with $$D_{\alpha}\phi:=sup_{a\in \alpha} D_{a}\phi =sup_{a\in \alpha} \sup_{x,y \in a} \dfrac{|\phi(x)-\phi(y)|}{d(x,y)}<\infty$$ and distribution $G$ in the domain of attraction of a stable law with order $1<d\leq 2$. Then $\{X_n:=\phi \circ T^{n-1}\}$ has a conditional local limit theorem with $B_n=n^{1/d}L(n)$, where $L(n)$ is a slowly varying function. By Theorem \ref{main}, the scaled local time of $S_n$ converges to Mittag-Leffler distribution strongly and (\ref{dev}) holds. 

If in addition, the return time process $R_n$ of the local time $\ell_n$ is uniformly or strongly mixing from below, then (\ref{sup}) holds.

\end{corollary}

In \cite{Jon3}, the conditions for  finite and countable state Markov chains and Markov interval maps to imply the  Gibbs-Markov property are listed. Applications of the foregoing results are straight forward, in particular to the behavior of partial sums when the Markov chain starts at independent values. 

%
%\begin{example}\label{Markov} Let $S$ be a countable set, $P: S\times S\to [0,1]$ be an aperiodic, irreducible stochastic matrix and $\pi$ is a probability distribution with $\pi_s>0 \; \forall\; s\in S$. Let $T: S^{\mathbb{N}}\to S^{\mathbb{N}}$ be the shift and define a Markovian probability $m$ on $S^{\mathbb{N}}$ by
%$$m([s_1,s_2,\cdots, s_n]):=\pi_{s_1}p_{s_1,s_2}p_{s_2,s_3}\cdots p_{s_{n-1}s_{n}}.$$
% 
%Let $\Omega:=\{S^{\mathbb{N}}: m([x_1,x_2,\cdots, x_n])>0 \;\forall\; n\geq 1\}$ and $\alpha=\{[s]:s\in S\}$. Then $(X, \mathcal{B}, m,T, \alpha)$ is a Markov map. In addition, suppose there exists $M>1$ such that 
%$$\frac{1}{M}\leq \frac{p_{s,t}}{\pi_t}\leq M, \text{ whenever } s, t \in \mathbb{Z}, p_{s,t}>0.$$
%This is an example of Gibbs-Markov map. $T$ is measure preserving and mixing. Suppose $\pi$ is the in domain of attraction of a stable law of order $p\in (0,2)$. Let $S=\mathbb{Z}$. Define $\phi: X\to Z$ by $\phi(s^{\mathbb{N}})=s_1$, it is Lipschitz continuous on each $a\in \alpha$ with $D_{\alpha}: \sup_{a}D_{a}\phi<\infty.$ \textcolor{red}{How to prove that $\phi$ is aperiodic? }
%
%The Markov chain has a conditional local limit theorem, hence the  local time $\ell_n:=\#\{i:S_i=0\}$ converges to Mittag-Leffler distribution after scaled and (\ref{dev}) holds for $\ell_n$. 
%\end{example}

Next, we show two examples of stationary processes whose local times converge to the Mittag-Leffler distribution.

\begin{example} [Continued Fractions]\label{cf}
Any irrational number $x\in (0,1]$ can be uniquely expressed as a simple non-terminating continued fraction $x=[0; c_1(x), c_2(x), \cdots]=: \frac{1}{c_1(x)+\frac{1}{c_2(x)+\frac{1}{c_3(x)+\cdots}}}$. The continued fraction transformation $T$ is defined by 
$$T(x)=x-[\frac{1}{x}].$$
Define $\phi: (0,1]\to \mathbb{N}$ by $\phi(x)=c_1(x)$ and  $X_n:=\phi\circ T^{n-1}$. We have the following convergence in distribution with respect to any absolutely continuous probability measure $m\ll \lambda$, where $\lambda$ is the Lebesgue measure, i.e.
$$\frac{\sum_{i=1}^n X_i}{n/\log 2}-\log n \to F,$$
where $F$ has a stable distribution (cf. eg. \cite{PW}).

Let $a_n:=\{x\in(0,1]: c_1(x)=n\}$ for every $n\in \mathbb{N}_+$ and  the partition is  $\alpha=\{a_n: n\in \mathbb{N}\}$. 
Then $(\Omega, \mathcal{B}, \mu, T, \alpha)$ is the continued fraction transformation where $\Omega=[0, 1]$. It is a mixing and measure preserving Gibbs-Markov map with respect to the Gauss measure $d\mu=\frac{1}{\ln 2}\frac{1}{1+x}dx$. Define the metric on $\Omega$ to be $d(x,y)=r^{\inf\{n: a_n(x)\neq a_n(y) \}}$, where $r\in (0,1).$ Note that $\phi$ is Lipschitz continuous on each partition. 

Define $(X,\mathcal{F},\nu, T_{X}, \beta)$ to be the direct product  of $(\Omega, \mathcal{B}, \mu, T, \alpha)$ with metric $d_X((x,y),(x',y'))=\max\{d(x,x'),d(y,y')\}.$ Then one can check that $(X,\mathcal{F},\nu, T_{X}, \beta)$  is still a mixing and measure preserving Gibbs-Markov map. Let $f: X\to \mathbb{Z}$ be defined by $f(x,y)=\phi(x)-\phi(y)$. Since $\phi$ is Lipschitz on partitions $\alpha$, so is $f$.  Define $Y_n((x,y))=f\circ T_X^{n-1}(x,y)= X_n(x)-X_n(y),  (x,y)\in X.$ $Y_n$ is in the domain of attraction of a stable law. Let $S_n:=\sum_{i=1}^n Y_i$ .The local time at level $0$ of $S_n$ is denoted to be $\ell_n(x,y)=\sum_{i=1}^n \mathbb 1_{\{S_i(x,y)=0\}}$.  By applying Corollary \ref{Gibbs-Markov} to the Gibbs-Markov map $(X,\mathcal{F},\nu, T_{X}, \beta)$ and the Lipschitz continuous function $f$, $S_n$ has a conditional local limit theorem and  the  local time $\ell_n$ converges to the Mittag-Leffler distribution after scaled and (\ref{dev}) holds for $\ell_n$. In particular this applies to the number of times that  the partial sum $\sum_{j\le i} \phi\circ T^j$ agree at times $i\le n$ when the initial values are chosen independently.

\end{example}

\begin{example}[$\beta$ transformation]\label{beta}
Fix $\beta>1$ and $T:[0,1]\to[0,1]$ is defined by $Tx:=\beta x$ mod $1$. Let $\phi: [0,1] \to \mathbb{Z}$ be defined as $\phi(x)=[\beta x]$ and $X_n(x)=\phi\circ T^{n-1}(x)=[\beta T^{n-1}x]$.  There exists an absolutely continuous invariant probability measure $P$. By \cite{Jon4}, there is a conditional local limit theorem for the partial sum $S_n$ of $\{X_n\}$. Then Theorem \ref{main} can be applied to $([0,1], \mathcal{B},  P, T)$ and $\{X_n\}$, it follows that the scaled local time of $S_n$ at level $E[\phi]$ converges to the Mittag-Leffler distribution and (\ref{dev}) holds when $E[\phi]$ is an integer. When $E[\phi]$ is not an integer, a similar product space as in Example \ref{cf} can be constructed and the same conclusion for the local time in the product space holds.
\end{example}
In the last part, we prove an almost sure weak convergence theorem of the local times. Almost sure central limit theorems were first introduced by Brosamler(1988) \cite{Brosamler} and Schatte (1988) \cite{Schatte}. 
It has been extended for several classes of independent and dependent random variables. For example, Peligrad and Shao (1995) \cite{PS} gave an almost sure central limit theorem for associated sequences, strongly mixing and $\rho$-mixing sequences under the same conditions that assure that the central limit theorem holds.  Gonchigdanzan (2002) \cite{Gon} proved an ASCLT for strongly mixing sequence of random variables under different conditions.  

However, the corresponding results  for the weak convergence  of local times are sparse. We only know that for aperiodic integer-valued random walks, Berkes, Istv{\'a}n and Cs{\'a}ki, Endre (2001) \cite{Uni} established an almost sure  limit theorem when $X_n$ is in the domain of attraction of a stable law of order $d\in(1,2]$. Here, we prove that an almost sure weak limit theorem holds for local times of stationary  processes when the local limit theorem holds in a stronger form than (\ref{clltdef}). 

\begin{definition}\label{cllt-exp}
An integer-valued  stationary process $(X_n)_{n\ge 1}$ is said to satisfy the $L^\infty$ conditional local limit theorem at $0$ if  there exists a sequence $g_n\in \mathbb R$ of real constants such that 
$$\lim_{n\to\infty} g_n=g(0)>0$$
and
$$ \| B_nP_{T^n}(S_n=x)-g_n\|_{{\infty}}$$
decreases  exponentially fast.
\end{definition}

This condition is  essentially stronger than condition (\ref{cllt}) holding  in $L^\infty(P)$ and the convergence is exponentially fast. 

\begin{theorem}[Almost sure central limit theorem for the local times] \label{ass3}
Let $(X_n)_{n\in \mathbb N}$ be an integer-valued stationary process satisfying the  local limit theorem at $0$ in Definition \ref{cllt-exp} and $B_n=n^{\beta}L(n)$ and slowly varying function $L(n)$ converges to $c>0$.   Moreover, assume that the following two conditions are satisfied:  for some constants $K>0$ and $\delta>0$ and for all bounded  Lipschitz continuous functions $g, F\in C_b(\mathbb R)$  and $x\in \mathbb Z$ it holds that
\begin{equation}\label{eq:cond1}
cov\left(g(\ell_k),F\circ T_{2k})\right) \le \left(\log\log k\right)^{-1-\delta}
\end{equation}
and
\begin{equation}\label{eq:cond2}
 \sum_{n=1}^\infty |E\left( \mathbb 1_{\{S_n=x\}}-\mathbb 1_{\{S_n=0\}}\right)|\le K(1+ |x|^{\frac\alpha{1-\alpha}}),
\end{equation}
where $\alpha=1-\beta.$
Then
\begin{equation}\label{1}
\lim_{N\to\infty} \dfrac{1}{\log N}\sum_{k=1}^N \dfrac{1}{k}\mathbb 1_{\{\frac{\ell_k}{a_k} \leq x\}}=M(x) \;\;a.s.
\end{equation}
is equivalent to

\begin{equation}\label{2}
\lim_{N\to\infty} \dfrac{1}{\log N}\sum_{k=1}^N \dfrac{1}{k}P\{\frac{\ell_k}{a_k} \leq x\}=M(x),
\end{equation}
where $M(x)$ is a cumulative distribution function.
\end{theorem}

\begin{corollary} [Gibbs-Markov maps]\label{GM}
 The almost sure central limit theorem for the local times holds under the same setting as Corollary \ref{Gibbs-Markov}.
\end{corollary}
It is because  the conditional local limit theorem in the sense of Definition \ref{cllt-exp} holds and the assumptions on the transfer operator in Section \ref{subsection3} are satisfied.

\begin{example}[$\beta$ transformation]
The stationary process $\{X_n\}$ defined by the $\beta$ transformation  in Example \ref{beta}  has the almost sure central limit theorem \ref{ass3}.
\end{example}

This paper is structured as follows. Section \ref{section2} is devoted to proving the limiting distribution of the local time  of $X_n$. In Section \ref{section3}, the almost sure central limit theorem of the local time is proved.

\section{Asymptotic Distribution of the Local Times $\ell_n$}\label{section2}

Since $(X_n)_{n\in \mathbb N}$ is stationary and integer-valued we may represend the process on the probability space $\Omega=\mathbb Z^{\mathbb N}$ equipped with the $\sigma$-algebra $\mathcal F$ generated by all coordinate projections $\phi_n:\Omega\to \mathbb Z$ and the probability measure $P$ so that the joint distribution of the $\phi_n$ agrees with that of the process $(X_n)$. If $T:\Omega\to \Omega$ denotes the shift transformation, then $\phi_n=\phi_1\circ T^{n-1}$ and hence the process $X_n$  can be written in the form $X_n=\phi\circ T^{n-1}$ where we use $\phi=\phi_1$.
Next  we extend the probability space $(\Omega,\mathcal{F},P)$ to a product space. Define $\tilde{T}: \Omega \times \mathbb{Z} \to \Omega \times \mathbb{Z}$ by $\tilde{T}(\omega, n)=(T \omega, n+\phi(\omega))$,
then by induction, $\tilde{T}^k(\omega, n)=(T^k \omega, n+S_k(\omega))$. Let
$m_{\mathbb{Z}}$ be the counting measure on the space  $\mathbb{Z}$, and $\mathcal{Z}$ be  the Borel-$\sigma$ algebra of $\mathbb{Z}$. A new dynamical system $(X, \mathcal{B}, \mu,\tilde{T} )$ then can be defined, where $X=\Omega \times \mathbb{Z}$, $\mathcal{B}=\mathcal{F}\otimes \mathcal{Z}$ and $\mu=P\otimes m_{\mathbb{Z}}$ is the product measure. We denote by $P_T$ and $P_{\tilde T}$ the  transfer operators of $T$ and $\tilde T$, respectively.

\begin{lemma}\label{point}
Suppose  $\{X_n\}$ has the conditional local limit theorem at $0$  (cf. (\ref{clltdef})). Then
\begin{enumerate}
\item $\tilde{T}$ is conservative and measure preserving in $(X,\mathcal{B},\mu)$.
\item There exists a probability space $(Y, \mathcal{C}, \lambda)$, and a collection of measures 
$\{\mu_{y}:y \in Y\}$ on $(X,\mathcal{B})$ such that 
\begin{enumerate}
\item For $y \in Y$, $\tilde{T}$ is a conservative ergodic measure-preserving transformation of $(X, \mathcal{B}, \mu_y)$. 
\item For $A\in \mathcal{B}$, the map $y\to \mu_{y}(A)$ is measurable and 
$$\mu(A)=\int_{Y} \mu_{y}(A)d\lambda(y).$$
\end{enumerate}
\item $\lambda$-almost surely for $y$, $(X, \mathcal{B}, \mu_y, \tilde{T})$  is pointwise dual ergodic.
\end{enumerate}
\end{lemma}

\begin{proof}
1. For any  $m\in \mathbb{Z}$, let $ f: \Omega \times \mathbb{Z}\to \mathbb{R} $ be defined as $f(\omega, n)=h(\omega)\otimes  \mathbb 1_{\{m\}}(n)$.

It can be proved that $\mu$ almost surely for $(\omega, k)\in X$,
\begin{equation}\label{TTildeT}
P_{\tilde T}^n  (h\otimes \mathbb 1_{\{m\}})(\omega, k)=P_{T}^n\Bigg(h(\cdot)  \mathbb 1_{\{m\}}\bigg(k-S_n(\cdot)\bigg)\Bigg)(\omega).
\end{equation}
Set $h\equiv \mathbb 1$, with the assumption of the local conditional limit theorem (\ref{clltdef}), 
\begin{eqnarray*}
\sum_{n=0}^{N}P_{{T}}^n \mathbb 1_{\{m\}}\bigg(k-S_n(\cdot)\bigg)(\omega)&=&\sum_{n=0}^{N} P(S_n=k-m|T^n()=\omega)\\
&\sim& \sum_{n=0}^{N}\dfrac{1}{B_n}g(0)=:a_N \text{ $P -a.s.$ for $\omega$}.
\end{eqnarray*}
Since $ \sum_{n=0}^{\infty}\dfrac{1}{B_n}=\infty$, it follows that 
\begin{eqnarray*}
\sum_{n=1}^{\infty}P_{\tilde{T}}^n  (\mathbb 1 \otimes  \mathbb 1_{\{m\}})=\infty \;\mu-a.s.
\end{eqnarray*}
By linearity of $P_{T}^n$, for $f(\omega, x)=\sum_{m\in\mathbb{Z}} k_m \mathbb 1_{\{m\}}(x)$ with $k_m>0$ with $\sum_{m\in\mathbb{Z}} k_m <\infty$, one has $0<f\in L^1(\mu)$  and 
\begin{eqnarray*}
\sum_{n=1}^{\infty}P_{\tilde{T}}^n  f=\infty, \;\mu-a.s..
\end{eqnarray*}
Proposition 1.3.1 in \cite{Jon1} states that 
\begin{eqnarray*}
\{(\omega, x): \sum_{n=1}^{\infty} P_{\tilde{T}}^nf=\infty \}=\mathcal{C} \text{ mod $\mu$ for any $f\in L^1(\mu), f>0$}
\end{eqnarray*}
where $\mathcal{C}$ is the conservative part of $\tilde{T}$.
Hence $\mathcal{C}=X$ mod $\mu$, which means that $\tilde{T}$ is conservative.

 2. 
The proof of the ergodicity decomposition is an adaption of the corresponding argument of Section 2.2.9 of \cite{Jon}(page 63).

3. Since 
\begin{eqnarray*}
\sum_{n=1}^{N}P_{\tilde{T}}^n  (\mathbb 1_{\Omega}\otimes \mathbb 1_{\{m\}})\sim a_N \;\mu-a.s.,
\end{eqnarray*}
one also has that relation  
$\mu_y$-a.s..
From 2, it is known that $\tilde{T}$ is conservative and ergodic on $(X,\mathcal{B}, \mu_y)$,
by  Hurewicz's ergodic theorem, one has $\forall f\in L^1(\mu_y)$, almost surely,
\begin{equation}
\dfrac{1}{a_n}\sum_{k=0}^n P_{\tilde{T}}^kf\sim \dfrac{\sum_{k=0}^n P_{\tilde{T}}^kf}{\sum_{k=0}^n P_{\tilde{T}}^k\bigg( \mathbb 1_{\Omega} \otimes \mathbb 1_{\{m\}} \bigg)}\to \dfrac{\int_{\Omega\times \mathbb{Z}}fd\mu_y}{\int_{\Omega\times \mathbb{Z}} \mathbb 1_{\Omega}\otimes \mathbb 1_{\{m\}} d\mu_y}.
\end{equation}

Since $a_N$ doesn't depend on $m$,  $\frac{1}{\int_{\Omega\times \mathbb{Z}} 1\otimes \mathbb 1_{\{m\}}d\mu_y}$ can be written as $C(y)$.
Hence, $(X, \mathcal{B}, \mu_y, \tilde{T})$  is pointwise dual ergodic with return sequence $a_n C(y)= C(y)g(0)\sum_{i=1}^n\dfrac{1}{B_i}$.\QEDA
\end{proof}

Next we prove the limiting distribution of the scaled local time $\frac{\ell_n}{a_n}$.   
\begin{proof} of the Theorem \ref{main}

Since $B_n=n^{\beta}L(n)$ is regularly varying of order $\beta$, by Karamata's integral theorem (c.f. e.g.\cite{Whitt} Theorem $A.9.$), $\{a_n\}_{n=1}^{\infty}$ is regularly varying of order $\alpha=1-\beta \in (0,\frac{1}{2}]$ and $a_n\sim n^{\alpha}\dfrac{C}{L(n)}$.

Let $A=\Omega \times \{0\}$,  and  define $S^{\tilde{T}}_n(f)(\omega,m):=\sum_{i=1}^n f\circ \tilde{T}^{i}(\omega,m).$ Since $\tilde{T}^n(\omega, k)=(T^n\omega, k+S_n(\omega))$, the local time of $\{S_n\}$ has the following representation: \\
\begin{equation}
\ell_n(\omega)=\sum_{i=1}^n \mathbb 1_{\{S_i(\omega)=0\}}=\sum_{i=1}^n \mathbb 1_{\{ A\}}(\tilde{T}^i(\omega, 0) )=S_n^{\tilde{T}}(\mathbb 1_{\{ A\}})(\omega, 0).
\end{equation}

From Lemma \ref{point}, $(X, \mathcal{B}, \mu_y,\tilde{T})$ is pointwise dual-ergodic.  Since $a_n$ is regularly varying, by applying  Theorem  1 in \cite{Jon}, for any $f\in L^1(\mu_y), f\geq 0,$ one has  strong convergence,  denoted by
\begin{equation}
\dfrac{S^{\tilde{T}}_n(f)}{a_n} \stackrel{\mathfrak{L}}{\longrightarrow} {C(y)}\mu_y(f)Y_{\alpha},
\end{equation}
which means
\begin{equation} 
\int_{X} g\bigg(\dfrac{S^{\tilde{T}}_n(f)(x,\omega)}{a_n}\bigg)h_y(x,\omega)d\mu_y(x,\omega)\to E[g({C(y)}\mu_y(f)Y_{\alpha})],
\end{equation}
for any bounded and continuous function $g$ and for any probability density function $h_y$ of $(X, \mathcal{B}, \mu_y)$. Here $Y_{\alpha}$ has the normalized  Mittag-Leffler distribution of order $\alpha=1-\beta$.

Let probability density function $H(\omega,m)$ of $(X, \mathcal{B}, \mu)$ be defined as 
$$
H(\omega,m)=
\begin{cases}
H(\omega), & m=0,\\
0, & m\neq 0,
\end{cases}
$$ where $H(\omega)$ is an arbitrary probability density function in $\Omega$. 
For each $y$, define 
 \begin{eqnarray}
 h_y(\omega,x)=
 \begin{cases}
 \frac{1}{\int_X H(\omega,x)d\mu_y} H(\omega,x), &\int_X H(\omega,x)d\mu_y\neq 0;\\
 0, &\int_X H(\omega,x)d\mu_y= 0.
 \end{cases}
 \end{eqnarray}
Then  $ h_y(\omega,x)$ is a probability density function on $(X,\mathcal{B}, \mu_y)$ for $y\in U$ where $U=\{y\in Y: \int_X H(\omega,x)d\mu_y\neq 0\}$.  
 
By the Disintegration Theorem (cf., eg.,\cite{Jon1}), 
\begin{eqnarray*}
&&\int_{X} g\bigg(\dfrac{S^{\tilde{T}}_n(f)(\omega,x)}{a_n}\bigg)H(\omega,x)d\mu(\omega, x)\\
&=& \int_U (\int_X H(\omega,x)d\mu_y)\int_X g\bigg(\dfrac{S^{\tilde{T}}_n(f)(\omega,x)}{a_n}\bigg)h_y(\omega,x)d\mu_y(\omega, x) d\lambda(y)\\
&=&\int_Y (\int_X H(\omega,x)d\mu_y)\int_X g\bigg(\dfrac{S^{\tilde{T}}_n(f)(\omega,x)}{a_n}\bigg)h_y(\omega,x)d\mu_y(\omega, x) d\lambda(y)\\
&\to &\int_Y  \mu_y(H)  E[g(C(y)\mu_y(f) Y_{\alpha})]d\lambda(y),\qquad \text{by the dominated convergence theorem.}
\end{eqnarray*}
Let $f=\mathbb 1_{\Omega}\times \mathbb 1_{\{m\}}$, then $C(y)\mu_y(f)=1$.
The result above becomes 
$$
\int_{\Omega}g\bigg(\dfrac{\ell_n}{a_n}\bigg)H(\omega)dP(\omega)\to  E[g(Y_{\alpha})]
$$
for any bounded and continuous function $g$ and any probability density function $H$ of $(\Omega, \mathcal{F}, P).$ \QEDA
\end{proof}

\section{Proof of Almost Sure Weak Convergence Theorem}\label{section3}

\subsection{Proof of Theorem \ref{ass3}}
In this section, we shall show that the local time $\ell_n$ of $\{X_n\}$  has an almost sure weak convergence  theorem  under the assumptions of Theorem \ref{ass3}. The following proposition will be used in the proof of  Theorem \ref{ass3}, so we state it below and the proof of it is in Section \ref{subsection2}.
\begin{proposition}\label{prop1}
\begin{equation*}\label{clm1}
\Var\bigg(\dfrac{1}{\log N}\sum_{k=1}^N\dfrac{1}{k}g(\dfrac{\ell_k}{a_k})\bigg) =O\bigg((\log\log N)^{( -1-\delta)}\bigg) 
\end{equation*}
for some $\delta>0,$ as $N\to \infty$, where $g$ is  any bounded Lipschitz  function with Lipschitz constant $1$.
\end{proposition}

Once Proposition \ref{prop1} is granted, the proof of Theorem \ref{ass3} can be proved using standard arguments. We sketch these shortly.
\begin{proof}[Proof of Theorem \ref{ass3}]
By the dominated convergence theorem, statement (\ref{1}) implies (\ref{2}) when  taking expectation. To prove the other direction, 
it is sufficient to prove that (see e.g. Lacey and Philipp, 1990 \cite{Lacey})
\begin{equation}\label{eqn}
\lim_{N\to\infty} \dfrac{1}{\log N}\sum_{k=1}^N \dfrac{1}{k} \xi_k=0, \;a.s.
\end{equation}
for any bounded Lipschitz  continuous  function $g$ with Lipschitz constant $1$, 
where $\xi_k:=g\left(\frac{\ell_k}{a_k}\right)-E\left[g\left(\frac{\ell_k}{a_k}\right)\right]$.

Taking $N_i=\exp \exp i^{\epsilon}$, for any $\epsilon >\dfrac{1}{1+\delta}$, then Proposition \ref{prop1} implies
\begin{equation} \label{eqn1}
\sum_{i=1}^{\infty}\dfrac{1}{\log ^2 N_i}E(\sum_{k=1}^{N_i}\dfrac{1}{k}\xi_k)^2<\infty.
\end{equation}
By Borel-Cantelli lemma, 
\begin{equation}
\lim_{i\to\infty} \dfrac{1}{\log N_i}\sum_{k=1}^{N_i} \dfrac{1}{k} \xi_k=0, \;a.s.
\end{equation}
 For any $N$, there exists $k$ such that $N_k\leq N<N_{k+1}$ and we have 
\begin{eqnarray*}
\dfrac{1}{\log N}|\sum_{j=1}^N\dfrac{1}{j}\xi_j|&\leq& \dfrac{1}{\log N_{k}}(|\sum_{j=1}^{N_k}\dfrac{1}{j}\xi_j|+\sum_{j=N_k+1}^{N_{k+1}} \dfrac{1}{j}|\xi_j|)\\
&\leq&\dfrac{1}{\log N_{k}}|\sum_{j=1}^{N_k}\dfrac{1}{j}\xi_j|+\dfrac{C}{\log N_{k}} (\log  N_{k+1}-\log N_k)\\
&\to &0 \text{ as $k\to \infty$, a.s.}
\end{eqnarray*}
The last step is because $((1+k)^{\epsilon}-k^{\epsilon})\to 0$ as $k\to \infty$ for any $\epsilon<1$,
 $\dfrac{\log N_{k+1}}{\log N_k}=e^{((1+k)^{\epsilon}-k^{\epsilon})}\to 1$ as $k\to \infty$.

Hence (\ref{eqn}) holds and the proof is done. \QEDA
\end{proof}

\subsection{Proof of Proposition \ref{prop1}}\label{subsection2}

Proposition \ref{prop1}  is a result of the following two lemmas. 
\begin{lemma}\label{lm1}
 Suppose $\{X_n\}$ has conditional local limit theorem \ref{cllt-exp}, then $E[\ell_n]=O(a_n),$ where $a_n=g(0)\sum_{i=1}^n\dfrac{1}{B_i}$.
\end{lemma}

\begin{proof}
Since the convergence in the conditional local limit theorem is in the sense of \ref{cllt-exp}, $B_nP(S_n=0)= B_nE\left(P(S_n=0| T^n)\right)\to g(0)$ as $n\to \infty$. So $$\qquad\qquad\qquad\qquad E(\ell_n)=\sum_{i=1}^nP(S_i=0)\sim a_n=g(0)\sum_{i=1}^n\dfrac{1}{B_i}.\qquad\qquad\qquad\qquad\qquad\qquad\qquad{\QEDA}$$ 
\end{proof}
\begin{lemma}\label{lm2}
When $j>2k,$ and $k,j \to \infty$,
\begin{equation}
E\left(\ell_j-\ell(j,S_{2k})-\ell_{2k}+\ell(2k,S_{2k})\right)^2=O(a_jE[|S_{2k}|^{\frac{\alpha}{1-\alpha}})], 
\end{equation}
where $\alpha\in (0,\frac{1}{2}]$ and $a_n=\sum_{i=1}^n\frac{g(0)}{B_i}.$
\end{lemma}
\begin{remark}
For i.i.d. case, by Kesten and Spizer (1979) \cite{KS}, it is known that $E(\ell(n,x)-\ell(n,y))^2\leq C|x-y|^{\frac{\alpha}{1-\alpha}}n^{\alpha}$ when ${X_n}$ is in the domain of attraction of a stable law of order $d=\frac{1}{1-\alpha}$. 
\end{remark}
\begin{proof}
\begin{eqnarray*}
&&E\left(|\ell_j-\ell(j,S_{2k})-\ell_{2k}+\ell(2k,S_{2k})|^2\right)\\
&=&E\bigg(\sum_{i=2k+1}^j\mathbb 1_{\{S_i=S_{2k}\}}-\mathbb 1_{\{S_i=0\}}\bigg)^2\\
&=&\sum_{x\in\mathbb{Z}}E\bigg[\bigg(\sum_{i=2k+1}^j\mathbb 1_{\{S_i=x\}}-\mathbb 1_{\{S_i=0\}}\bigg)^2\mathbb 1_{\{S_{2k}=x\}}\bigg]\\
&\leq&\sum_{2k+1 \leq j_1\leq j_2 \leq j}2\bigg|\sum_{x\in\mathbb{Z}}E\left(\mathbb 1_{\{S_{j_1}=x\}}(\mathbb 1_{\{S_{j_2}=x\}}-\mathbb 1_{\{S_{j_2}=0\}})\mathbb 1_{\{S_{2k}=x\}}\right)\bigg|\\
&+&\sum_{2k+1 \leq j_1\leq j_2 \leq j}2\bigg|\sum_{x\in\mathbb{Z}}E\left(\mathbb 1_{\{S_{j_1}=0\}}(\mathbb 1_{\{S_{j_2}=x\}}-\mathbb 1_{\{S_{j_2}=0\}})\mathbb 1_{\{S_{2k}=x\}}\right)\bigg|.
\end{eqnarray*}
Due to the similar form of the two terms above,  let $z_1\in \{x,0\}$.
\begin{eqnarray*}
&&E\left(\mathbb 1_{\{S_{j_1}=z_1\}}(\mathbb 1_{\{S_{j_2}=x\}}-\mathbb 1_{\{S_{j_2}=0\}})\mathbb 1_{\{S_{2k}=x\}}\right)\\
&=&E\left(\mathbb 1_{\{S_{j_1-2k}\circ T^{2k}=z_1-x\}}(\mathbb 1_{\{S_{j_2-j_1}\circ T^{j_1}=x-z_1\}}-\mathbb 1_{\{S_{j_2-j_1}\circ T^{j_1}=-z_1\}})\mathbb 1_{\{S_{2k}=x\}}\right)\\
&=& E\left(E\left(\mathbb 1_{\{S_{j_1-2k}\circ T^{2k}=z_1-x\}}\mathbb 1_{\{S_{2k}=x\}}| T^{-{j_1}}\mathcal{F}\right)(\mathbb 1_{\{S_{j_2-j_1}\circ T^{j_1}=x-z_1\}}-\mathbb 1_{\{S_{j_2-j_1}\circ T^{j_1}=-z_1\}})\right)\\
&=& E\left(P_{T^{j_1}}(\mathbb 1_{\{S_{j_1-2k}\circ T^{2k}=z_1-x\}}\mathbb 1_{\{S_{2k}=x\}} )(\mathbb 1_{\{S_{j_2-j_1}=x-z_1\}}-\mathbb 1_{\{S_{j_2-j_1}=-z_1\}})\right).
\end{eqnarray*}
In order to bound this expression consider
% where the property P_T(f g\circ T)=g P_T f is used.
$$ B_{j_1-2k}B_{2k}P_{T^{j_1}}(\mathbb 1_{\{S_{j_1-2k}\circ T^{2k}=z_1-x\}}\mathbb 1_{\{S_{2k}=x\}} ) = B_{j_1-2k}P_{T^{j_1-2k}}\left(\mathbb 1_{\{S_{j_1-2k}=z_1-x\}}B_{2k} P_{T^{2k}}(\mathbb 1_{\{S_{2k}=x\}}\right),$$ 
which by the assumption of the $L^\infty$-conditional local limit theorem at $0$,  can be written in the form $ g_{j_1-2k}g_{2k} +Z$, $g_{j_1-2k}g_{2k}$ converging to $g(0)^2$ and $Z$ being a $L^{\infty}$ random variable with $\|Z\|_\infty \le c\theta^{j_1-2k}$. Then, for fixed $j_2-j_1$, and since $B_{2k}P(S_{2k}=x)$ is universally bounded, and using the assumption,
$$ E\left(E\left(\mathbb 1_{\{S_{j_1-2k}\circ T^{2k}=z_1-x\}}\mathbb 1_{\{S_{2k}=x\}}|T^{-{j_1}}\mathcal{B}\right)(\mathbb 1_{\{S_{j_2-j_1}\circ T^{j_1}=x-z_1\}}-\mathbb 1_{\{S_{j_2-j_1}\circ T^{j_1}=-z_1\}})\right)$$
 is bounded by 
$$ C P(S_{2k}=x) \left(\frac{1}{B_{j_1-2k}} q(j_2-j_1,x)+  \theta^{j_1-2k}\frac 1{B_{j_2-j_1}}\right)$$
for some constants $C>0$, where $\sum_{j_2-j_1=1}^\infty q(j_2-j_1,x)\le K x^{\frac{\alpha}{1-\alpha}}$ by (\ref{eq:cond2}). 
Summing over $x, z_1$, then over $j_2$ and finally over $j_1$ shows the lemma.

\QEDA
\end{proof}

\begin{proof}{ of Proposition \ref{prop1}}

Split  $\Var\bigg(\sum_{k=1}^N\dfrac{1}{k}g(\dfrac{\ell_k}{a_k})\bigg)$ up into three parts: $T_1, T_2$ and $T_3$:
\begin{eqnarray*}
&&\Var\bigg(\sum_{k=1}^N\dfrac{1}{k}g(\dfrac{\ell_k}{a_k})\bigg) =E[(\sum_{k=1}^N \dfrac{1}{k} \xi_k)^2]\\
&=&\sum_{k=1}^N \dfrac{1}{k^2} E[\xi_k^2]+2\sum_{1\leq k<j\leq N, 2k\geq j}\dfrac{|E[\xi_k \xi_j]|}{kj}+2\sum_{1\leq k \leq 2k<j\leq N}\dfrac{|E[\xi_k \xi_j]|}{kj}\\
&=&T_1+T_2+T_3.
\end{eqnarray*}

For $T_1$, since $\xi_k$ is bounded, there is a constant $C_1$ such that $T_1 \le C_1(\log N)^2$  for all $N\in \mathbb N$. For $T_2$, there is a constant $C_2$ such that
$$T_2\leq \|g\|_{\infty}\sum_{1\leq k<j\leq N, 2k\geq j}\dfrac{1}{kj}\leq C_2(\log N)^2.$$\\

For $T_3$, since $1\leq k \leq 2k< j\leq N$, let $$f_{(2k,j)}:=\dfrac{1}{a_j} \sum_{i=1}^{j-2k} \mathbb 1_{\{X_{2k+1}+...+X_{2k+i}=0\}}=\dfrac{1}{a_j}(\ell(j,S_{2k})-\ell(2k,S_{2k})),$$
which is  measurable with respect to $\mathcal{F}_{2k+1}^{j}=\sigma(X_{2k+1},...X_{j})$. 
Then
\begin{eqnarray*}
E[\xi_k \xi_j]
&=&cov \left(g(\frac{\ell_k}{a_k}),g(\frac{\ell_j}{a_j})\right)\\
&=&cov\left(g(\frac{\ell_k}{a_k}),g(\frac{\ell_j}{a_j})-g(f_{(2k,j)})\right)+cov\left(g(\frac{\ell_k}{a_k}),g(f_{(2k,j)})\right).
\end{eqnarray*}

Due to the assumption \ref{eq:cond1}, 
$$cov\left(g(\frac{\ell_k}{a_k}),g(f_{(2k,j)})\right)  \leq C (\log \log k)^{-1-\delta} =: C\alpha(k).$$

Because $g$ is  Lipschitz and bounded, $cov\left(g(\frac{\ell_k}{a_k}),g(\frac{\ell_j}{a_j})-g(f_{(2k,j)})\right)\leq CE\left[\left|\dfrac{\ell_j}{a_j}-f_{(2k,j)}\right|\right]$. 
\begin{eqnarray*}
E\left[\left|\dfrac{\ell_j}{a_j}-f_{(2k,j)}\right|\right]
&=&1/a_jE[|(\ell_j-\ell(j,S_{2k}))+\ell(2k,S_{2k})|]\\
&\leq&1/a_jE\left[|(\ell_j-\ell(j,S_{2k}))+\ell(2k,S_{2k})-\ell_{2k}|\right]+1/a_jE[\ell_{2k}].
\end{eqnarray*}

So when $1\leq k \leq 2k< j\leq N$, one has
\begin{eqnarray*}
&&E[\xi_k \xi_j] \\
&\leq& C_1E\left[\left|\dfrac{\ell(0,j)}{a_j}-f_{(2k,j)}\right|\right]+C_2\alpha(k)\\
&\leq& C_1\dfrac{1}{a_j} \bigg(E[|\ell_j-\ell(j,S_{2k})-\ell_{2k}+\ell(2k,S_{2k})|]+E[|\ell_{2k}|]\bigg)
+C_2\alpha(k).\\
\end{eqnarray*}

When
 $2k< j$, Lemma \ref{lm1}, Lemma \ref{lm2} and Jensen's inequality imply that as $k,j \to \infty$, 
 \begin{eqnarray*}
E[\xi_k\xi_j]
&\leq& C_1\dfrac{1}{a_j} \bigg(E[|\ell_j-\ell(j,S_{2k})-\ell_{2k}+\ell(2k,S_{2k})|]+E[|\ell_{2k}|]\bigg)
+C_2\alpha(k)\\
&\leq& C_1\dfrac{1}{a_j}\bigg(\left(E[|S_{2k}|^{\frac{\alpha}{(1-\alpha)}}]\right)^{\frac{1}{2}}a_j^{\frac{1}{2}}+a_{2k}\bigg)+C_2\alpha(k)\\
&\sim&
 C_1\dfrac{1}{a_j}\bigg({B_{2k}}^{\frac{\alpha}{2(1-\alpha)}}a_j^{\frac{1}{2}}+a_{2k}\bigg)+C_2\alpha(k).
\end{eqnarray*}

%In the last step,  $E[S_n^2]\sim B_n^2$ is used. % \textcolor{red}{Is it true? Do we need the convergence uniformly?}

 Since $B_n=n^{\beta}L(n)$ and $a_n\sim n^{\alpha}\dfrac{1}{L(n)}$ with $\alpha=1-\beta$, one has
\begin{eqnarray*}
T_3&=&\sum_{1\leq k\leq N}\sum_{2k< j\leq N}\dfrac{1}{kj}E(\xi_k \xi_j)\\
&\leq& \sum_{1\leq k\leq N}\sum_{2k< j\leq N}C\dfrac{1}{kj}\bigg((\dfrac{2k}{j})^{\frac{\alpha}{2}}\dfrac{L(2k)^{\frac{\alpha}{2(1-\alpha)}}}{L(j)^{-1/2}}+(\dfrac{2k}{j})^{\alpha}\dfrac{L(j)}{L(2k)}+\alpha(k)\bigg).
\end{eqnarray*}
Since we assume $L(n)\to c$,
\begin{eqnarray*}
T_3&\leq& \sum_{1\leq k\leq N}\sum_{2k< j\leq N}C\dfrac{1}{kj}(\dfrac{k}{j})^{\frac{\alpha}{2}}+\sum_{1\leq k\leq N}\sum_{2k< j\leq N} C\dfrac{1}{kj}\alpha(k)\\
&=&T_{31}+T_{32}.
\end{eqnarray*}

$T_{31}\leq C\sum_{1\leq j \leq N}\dfrac{1}{j^{1+\alpha/2}}\sum_{1\leq k< j}\dfrac{1}{k^{1-\alpha/2}}= 
O(\log N)$.

And since $\alpha(k)=O((\log \log k)^{-1-\delta})$, by integration by parts and that $\dfrac{x}{(\log x)^{2+\delta}}$ is an increasing function of $x$ when $x$ is big enough, 
\begin{eqnarray*}
T_{32}\leq C\sum_{k=1}^N \dfrac{\alpha(k)}{k}\sum_{j=k}^{N} \dfrac{1}{j}&\leq&C\sum_{k=1}^N \dfrac{\log N}{k (\log \log k)^{1+\delta}}=O(\log^2 N (\log\log N)^{-(1+\delta)}).
\end{eqnarray*}

So 
$T_3=O((\log\log N)^{-1-\delta}\log^2N),$ as $N\to \infty$. Hence Proposition \ref{prop1} is proved.\QEDA
\end{proof}

\section{Transfer operators}\label{subsection3}

We turn to the investigation  of conditions (\ref{eq:cond1}) and (\ref{eq:cond2}) and show that they can be derived from the theory of transfer operators in dynamics.
Define the characteristic function operator $P_t: L^1(P)\to L^1(P)$ by $P_tf:=P_{T}(e^{it\phi}f)$, which is the perturbation of $P_T$. By induction, $P^n_t f=P_{T^n}(e^{itS_n}f)$. Let the space $\mathcal{L}$ be the subspace in $L^1(P)$ of all functions with norm: $\|f\|:=\|f\|_{\infty}+D_f$ where $D_f$ is the Lipschitz constant of $f$.  
 We assume  that $P_t$ acts on $\mathcal L$ and has the following properties:
 \begin{itemize}
\item There exists $\delta>0$, such that when $t\in C_{\delta}:=[-\delta, \delta]$, $P_t$ has a representation: $P_t=\lambda_t\pi_t+N_t$, $\pi_t N_t=N_t\pi_t=0,$ and the $\pi_t$ is a one-dimensional projection generated by an eigenfunction $V_t$ of $P_t$, i.e. $P_t V_t=\lambda_t V_t.$  It implies that $P_t^n=\lambda_t^n\pi_t+N_t^n$. 
\item There exists constants $K, K_1$ and $\theta_1<1$ such that on $C_{\delta}$, $\|\pi_t\|\leq K_1$,  $\|N_t\|\leq \theta_1<1$, $|\lambda_t|\leq 1-K|t|^d$.
\item There exists $\theta_2<1$ such that  for $|t|>\delta$, $\|P_t\|\leq \theta_2<1.$ 
\item $\phi $ is Lipschitz continuous.
\end{itemize}
We end up this article by proving conditions  \ref{eq:cond1} and \ref{eq:cond2} under these assumptions. This also will complete the proof of the corollary \ref{GM} in Section 1 since from \cite{Jon3}, one can see that Gibbs-Markov maps satisfy all the assumptions  above. An example not satisfying the above condition can be derived for functions of the fractional Brownian motion as in \cite{DeZh}.

\begin{proof}[Proof of (\ref{eq:cond1})]

Let $g(\frac{\ell_k}{a_k}):=\int_{\Omega}g(\frac{\ell_k}{a_k})dP+\hat{g}:=C_k+\hat{g}$, then by $P_{2k}^T=P+N^{2k}$, 
\begin{eqnarray*}
&&cov \left(g(\frac{\ell_k}{a_k}),F\circ T^{2k})\right)\\
&=&\int_{\Omega} g(\frac{\ell_k}{a_k}) (F\circ T^{2k})dP-\int_{\Omega} g(\frac{\ell_k}{a_k})dP\int_{\Omega}  F\circ T^{2k} dP\\
&=&\int_{\Omega} P^{2k}_T\left(g(\frac{\ell_k}{a_k})\right) F dP-\int_{\Omega} g(\frac{\ell_k}{a_k})dP\int_{\Omega}  F dP\\
&=&\int_{\Omega} N^{2k}(\hat{g}) F dP\\
&\leq&\|N^{2k}(\hat{g})\|\|F\|_1\\
&\leq&C\theta_1^{2k}\|g\|\\
&\leq&C(\log \log (2k))^{-1-\delta}.
\end{eqnarray*}

\end{proof}

\begin{proof}[Proof of  (\ref{eq:cond2})]

Let  $\zeta=(e^{-itx}- 1)$.  Then
\begin{eqnarray*}
E(\ell(n,x)-\ell_n) &=&\sum_{1\leq j \leq n}E(\mathbb 1_{\{S_{j}=x\}}-\mathbb 1_{\{S_{j}=0\}}) \\
=\sum_{1\leq j \leq n} \zeta E\left(\int_{[-\pi,\pi]} e^{itS_{j}} ~dt\right)
&=&\sum_{1\leq j \leq n} \int_{[-\pi,\pi]} \zeta E[P_{t}^{j}\mathbb 1]dt
\end{eqnarray*}
and 
\begin{eqnarray*}
&&\int_{[-\pi,\pi]} \zeta E[P_{t}^{j} \mathbb 1]dt=Re\int_{[-\pi,\pi]} \zeta E[P_{t}^{j}\mathbb 1] dt\\
&&=Re\int_{C_{\delta}} \zeta \lambda^{j}_{t}E[\pi_t \mathbb 1]dt  + Re\int_{C_{\delta}} \zeta E[N_{t}^{j} \mathbb 1 dt + Re\int_{\bar{C_{\delta}} }\zeta E[P_{t}^{j} \mathbb 1]dt.
\end{eqnarray*}

On $C_{\delta}$, one has $\|N_t^n\|\leq \theta_1^n$ where $|\theta_1|<1$, whence 
\begin{eqnarray*}
\left|\sum_{j=1}^n Re\int_{C_\delta} \zeta E[N_t^j] dt\right|
&\le &\sum_{j=1}^n\left|\int_{C_{\delta}} Re\left[\zeta E[N_{t}^{j} \mathbb 1]\right]dt\right|\\
&\leq& 2\sum_{j=1}^n\int_{C_{\delta}}   \|N_{t}^{j}\|dt
\leq 4\delta \frac {1}{1-\theta_1}.
\end{eqnarray*}
On $\bar{C_\delta}$, one has $\|P_t^j\|\le \theta_2$ so that
\begin{eqnarray*}
\left|\sum_{j=1}^n Re\int_{\bar{C_\delta}} \zeta E[P_t^j] dt\right|
&\le &\sum_{j=1}^n\left|\int_{\bar{C_\delta}} Re\left[\zeta E[P_{t}^{j} \mathbb 1]\right]dt\right|\\
&\leq& 2\sum_{j=1}^n\int_{\bar{C_\delta}}   \|P_{t}^{j}\|dt
\leq 4\pi \frac {1}{1-\theta_2}.
\end{eqnarray*}

The main part is the first term.  We show that for some constant $C>0$
\begin{equation*} \left|\sum_{j=1}^nRe \int_{C_{\delta}} \zeta \lambda^{j}_{t}\pi_t 1dt\right|\leq  C|x|^{\frac{\alpha}{1-\alpha}}dt
\end{equation*}

Let $\pi_t \mathbb 1= W_t$, so $\|W_t\}\le \|\pi_t\|\le  K_1$. Then
\begin{equation*}
\left|\sum_{j=1}^nRe\int_{C_{\delta}} \zeta \lambda^{j}_{t} W_tdt\right|
\leq \sum_{j=1}^n| \int_{C_{\delta}} Re(\lambda_{t}^{j}W_t)\bigg(\cos(xt)-1\bigg) W_tdt|+\sum_{j=1}^n|\int_{C_{\delta}} Im(\lambda_{t}^{j}W_t)\sin(xt) dt|.
\end{equation*}

Since $|Re \lambda_{t}^{j_2}W_t|\leq K_1|\lambda_{t}^{j_2}|\leq (1-Ct^d)^{j}$,  we have 
\begin{equation*}
\sum_{j=1}^n \left|\int_{C_{\delta}} Re(\lambda_{t}^{j}W_t)\bigg(cos(xt)-1)\bigg) dt\right|
\leq K_1C^{-1}\int_{[-\delta, \delta]} \frac{1}{|t|^d} |\cos (xt)-1)|~dt.
\end{equation*}
If $|\delta x|<1$, then $(1-\cos tx)\leq |tx|^2\leq |tx|^d,$ so
\begin{eqnarray*}
&&\int_{[-\delta, \delta]} \frac{1}{|t|^d} |1-\cos tx|~dt
\leq 2\int_{[0, 1/x]} \frac{1}{(t)^d} |tx|^d~dt=2~ |x|^{d-1}.
\end{eqnarray*}
If $|\delta x|\geq 1$, then
\begin{eqnarray*}
&&\int_{[-\delta, \delta]} \frac{1}{|t|^d} |1-\cos tx|~dt
\leq 2\int_{[0, 1/x]} \frac{1}{t^d} |tx|^2~dt+2\int_{[ 1/|x|,\infty]} \frac{1}{(t)^d} ~dt\leq C'~ |x|^{d-1}.
\end{eqnarray*}
for some constant $C'$.

Next we consider the second part  $\sum_{j_2}\left|\int_{C_{\delta}} Im(\lambda_{t}^{j}W_t)\sin(xt) dt\right|$.
\begin{eqnarray*}
\sum_{j=1}^n\left|\int_{C_{\delta}} Im(\lambda_{t}^{j}W_t)\sin(xt) dt\right|
&\leq& C^{-1}\int_{-{\delta}}^{{\delta}}\frac{1}{|t|^d}|\sin(xt)| dt\\
&\leq& 2C ^{-1}|x|^{d-1}\int_{0}^{{|x|\delta}}\frac{1}{u^d}|\sin(u)| du\\
\end{eqnarray*}

If $|x|\delta \leq 1$, then
\begin{eqnarray*}
\int_{0}^{{|x|\delta}}\frac{1}{u^d}|\sin(u)| du&\leq& \int_{0}^{1}u^{1-d} du<\infty.
\end{eqnarray*}

If $|x|\delta \geq 1$, then
\begin{eqnarray*}
\int_{0}^{{|x|\delta}}\frac{1}{u^d}|\sin(u)| du&\leq& \int_{0}^{1}u^{1-d} du+\int_{1}^{\delta |x|}\frac{1}{u^d} du <\infty.
\end{eqnarray*}

So 
\begin{equation*}
\sum_{j}^n\left|\int_{C_{\delta}} Im(\lambda_{t}^{j}W_t)\sin(xt)dt\right|
\leq C' |x|^{d-1}
\end{equation*}
for some constant $C'$ is proved.
\QEDA
\end{proof}

\end{document}